\documentclass[11pt, reqno]{amsart}

\usepackage{amscd,amsthm,amsfonts,amssymb,amsmath}
\usepackage{amsmath,tikz-cd}
\usepackage[colorlinks=true]{hyperref}
\usepackage{fullpage}
\usepackage[all]{xy}
\usepackage[small,nohug,heads=vee]{diagrams}

\usepackage{hyperref,dsfont}
\usepackage{color}
\hypersetup{
    colorlinks=true,
    citecolor=black,
    linkcolor=black,
  urlcolor=blue,
}

     \newcommand{\BF}{{\mathbb {F}}}
     
    \newcommand{\BI}{{\mathbb {I}}}

    \newcommand{\BQ}{{\mathbb {Q}}}

     \newcommand{\BZ}{{\mathbb {Z}}}

    \newcommand{\fm}{{\mathfrak{m}}} 
     \newcommand{\fp}{{\mathfrak{p}}}

    \newcommand{\ab}{{\mathrm{ab}}}\newcommand{\Ad}{{\mathrm{Ad}}}

    \newcommand{\Br}{{\mathrm{Br}}}

    \newcommand{\Cor}{{\mathrm{Cor}}}

    \newcommand{\End}{{\mathrm{End}}}

    \newcommand{\Gal}{{\mathrm{Gal}}} \newcommand{\GL}{{\mathrm{GL}}}
    
    \newcommand{\Hom}{{\mathrm{Hom}}}
    
    \renewcommand{\Im}{{\mathrm{Im}}}
    \newcommand{\Ind}{{\mathrm{Ind}}}

    \newcommand{\ord}{{\mathrm{ord}}}

    \renewcommand{\mod}{\ \mathrm{mod}\ }

    \newcommand{\Res}{{\mathrm{Res}}}

    \newcommand{\tor}{{\mathrm{tor}}}
    
    \newcommand{\ur}{{\mathrm{ur}}}

    \font\cyr=wncyr10

    \newcommand{\Sha}{\hbox{\cyr X}}

    \newcommand{\ov}{\overline}

    \theoremstyle{plain}
    \newtheorem{thm}{Theorem}[section] \newtheorem{cor}[thm]{Corollary}
    \newtheorem{lem}[thm]{Lemma}  \newtheorem{prop}[thm]{Proposition}

\theoremstyle{remark} \newtheorem{remark}[thm]{Remark}
\theoremstyle{remark} 
\theoremstyle{remark} 

    \numberwithin{equation}{section}

\begin{document}
\title{Ordinary deformations are unobstructed in the cyclotomic limit
}

\author{Ashay Burungale and Laurent Clozel }
\address{Ashay A. Burungale:  California Institute of Technology,
1200 E California Blvd, Pasadena CA 91125  And
The University of Texas at Austin, Austin, TX 78712, USA.} 
\email{ashayburungale@gmail.com}

\address{Laurent Clozel: Math\'ematiques Universit\'e Paris-Sud 91405 Orsay France} \email{laurent.clozel@math.u-psud.fr}

\begin{abstract}
The deformation theory of ordinary  representations of the absolute Galois groups of totally real number fields (over a finite field $k$) has been studied for a long time, starting with the work of Hida, Mazur and Tilouine, and continued by Wiles and others. Hida has studied the behaviour of these deformations when one considers the $p$-cyclotomic tower of extensions of the field. In the limit, one obtains a deformation ring $R_\infty$ classifying the ordinary deformations of the (Galois group of) the $p$-cyclotomic extension. 
We show that if $R_\infty$ is Noetherian 
and certain adjoint $\mu$-invariants vanish (as is often expected), then $R_\infty$ is  free over the ring of Witt vectors of $k$.

\end{abstract}
\maketitle
{
\hypersetup{linkcolor=black}
\tableofcontents
}

\section{Introduction}\label{s:Intro}

\subsection{Setup}
Let $p$ be an odd prime. Let $F$  be a totally real field of degree $d$ over $\mathbb{Q}$, unramified at $p$. All extensions of $F$ are contained in a fixed algebraic closure.  Let $F_{\infty}$ be the cyclotomic $\mathbb{Z}_p$-extension of $F$, and $F_n  \subset F_{\infty}$ the subextension of degree $p^n$. Thus $F_0=F$. Note that $F$ (and therefore $F_n$) does not contain the $p$-th roots of unity.

We write $\mathfrak{p}$ for a prime of $F$ dividing $p$. 
Since $F$ is unramified at $p$, we have

\begin{itemize}
\item[(ram)]   $F_n/F_0$ is totally ramified at $\mathfrak{p}$.
\end{itemize}
\vskip2mm

\vskip2mm

Let $S$ be a finite set of places of $F$, containing the infinite and $p$-adic places,  
and let $F_S$ be the maximal extension of $F$ unramified outside $S$; 
ditto $F_{n,S}$. We define $\Gamma_0 = \Gal (F_S /F)$ and similarly $\Gamma_n = \Gal( F_{n,S} / F_n)$.
\vskip2mm

In this setting, given an ordinary residual representation $\bar{\rho} :\Gamma_0 \rightarrow \GL_{2}(k)$ for $k$ a finite field of characteristic $p$ (cf. \S\ref{ss:ord}) one has the ordinary deformation ring $R_n$ 
of   $\bar{\rho} \vert _{\Gamma_n}$, 
classifying weight two ordinary deformations of   $\bar{\rho} \vert _{\Gamma_n}$ 
unramified outside $S$. It has been first studied by Hida \cite{Hi}.  One expects the size of $R_n$ to grow as $n\rightarrow \infty.$ We can form the inverse limit $R_{\infty} =  \varprojlim R_n.$  Suitably interpreted (below), it is the ordinary deformation ring of $\bar{\rho} \vert _{F_{\infty}}$. Our goal is to show that, under certain natural assumptions, such ordinary deformations are unobstructed: 
 $$R_{\infty} \cong W(k)[\![X_1, \dots, X_s]\!]$$ 
 for $W(k)$ the Witt ring and $s\geq 1$ an integer. 
Theorem \ref{main} is our main result. The assumptions are  $R_{\infty}$ is Noetherian, and certain adjoint $\mu$-invariants vanish (see \S\ref{ss:WL-II}).

In general, the obstructions are measured by the second adjoint Galois cohomology. 
Note that the $p$-cohomological dimension of $F_{\infty}$ is 1, cf. Serre \cite[Ch.2, Prop. 9]{Se1}. (Recall that primes of $F$ over $p$ are totally ramified in $F_{\infty}$, and that primes not dividing $p$ are inert, at least after a finite extension $F_n$ of $F$.) So, without the `ordinary' condition, the deformations are unobstructed 
 over $F_\infty$. 
 The corresponding deformation ring is however non-Noetherian. 
In contrast the ordinary deformation ring $R_\infty$ is expected to be often Noetherian and well-controlled (cf.~Hida's non-abelian Leopoldt conjecture \cite{Hi-LC}).  
To investigate whether it is smooth, one needs appropriately to account for the ordinary condition, which could yield obstructions. Much of our work will consist in proving the vanishing of the relevant $H^2$'s over $F_\infty$. There will be two main steps: a calculation of tangent spaces for infinite level local deformation problems (cf.~section~\ref{s:local})and a weak Leopoldt-type result (cf.~section~\ref{s:WL_ad}). 
The latter relies on the finiteness of the adjoint Bloch-Kato Selmer groups over $F_n$ (due to Allen ~\cite{Al}), and is also closely related to 
 the adjoint $\mu$-invariants.

\subsection{Context} 
Following Hida's discovery of $p$-adic families of modular forms (cf.~\cite{Hi'},~\cite{Hi0}), Mazur \cite{Ma0} introduced 
Galois deformation theory 
in the mid 80's. 
It has a rich history (cf.~\cite{Wi}), and continues to be fundamental to the study of Galois representations 
and their arithmetic. 
Iwasawa theory of deformation rings was initiated by Hida in the late 90's 
(cf.~\cite{Hi2},~\cite{Hi}).
It arose in the context of Iwasawa theory of the adjoint of a $p$-adic family of modular forms. 

The problem of the growth of deformation rings in the cyclotomic tower has been posed by Hida 
\cite[pp.~354--357]{Hi}. 
He proved that the vanishing of an adjoint $\mu$-invariant implies $R_\infty$ is Noetherian (cf.~\cite[Cor.~5.11]{Hi}).
The mysterious invariant $s\geq 1$ encodes the growth. In \cite{BCM} we will provide examples with $s >1$ for $\bar{\rho}$ verifying suitable conditions, and for a large set of ramification $S$.
One may seek arithmetic significance of the invariant $s$, such as its link with the adjoint Iwasawa theory. It is especially instructive to consider the residually CM case, which may lead to link with CM Iwasawa theory (cf.~\cite{MT},~\cite{HT}). 
Another basic problem is to explore connections with infinite level modular forms introduced in \cite{Cl1}, \cite{Cl2}.

As for the assumptions in our main theorem, 
it is expected that the $\mu$-invariant typically vanishes 
if the underlying Galois representation is residually irreducible (cf.~\cite{Su}). We are not aware of any general result towards it.
 Nevertheless, Remark \ref{Ex} (2)  
  presents some examples which illustrate the main theorem. The vanishing of the $\mu$-invariant seems critical (following Perrin-Riou) for Proposition \ref{mod p-WL-abs}.

We may ask\footnote{Tilouine and Urban have recently announced such a generalisation.} if the main result can be proved for $^{L}G$-valued deformations of a $^{L}G$-valued mod $p$ Galois representation with $G$ a reductive group. To follow the current approach, it seems essential to impose adequacy for the image of the mod $p$ Galois representation and suppose the vanishing of certain adjoint $\mu$-invariants. We remark that a key input in the current approach due to Allen \cite{Al} is already available for $G=\GL_{d}$.

\vskip2mm

\textit{Acknowledgements}

\vskip2mm

This work was begun by one of us (LC) in 2015, in collaboration with Akshay Venkatesh. Although he contributed a large part of its content, Venkatesh has declined to sign the final version. We wish to thank him for the impetus to this work. 

We are grateful to the referee for valuable comments and suggestions. 
We also thank Patrick Allen, Gebhard Boeckle, Haruzo Hida, Chandrashekhar Khare, Barry Mazur, Richard Taylor and Jacques Tilouine for helpful exchanges.
\vskip5mm

\textit{Notations}

\vskip2mm

Let $F_{n,\mathfrak{p}}$ be the localisation of $F_n$  at the unique prime above $\mathfrak{p}$. When $\mathfrak{p}$ is understood we will write $K_n := F_{n,\mathfrak{p}}$. Thus $[K_n : K_0] = p^n$.

We set $\Delta_n= \Gal(F_n /F) \cong \mathbb{Z}/p^n \mathbb{Z} $ and $\Delta_\infty = \varprojlim \Delta_n \cong \mathbb{Z}_p$. Also put
$$ \Omega = \varprojlim k[\Delta_n] \cong k[\![T]\!]$$
for the (modular) Iwasawa algebra, where $k$ is a finite field of characteristic $p$, and
$$ \Lambda = \varprojlim \mathbb{Z}_p[\Delta_n] \cong  \mathbb{Z}_p [\![T]\!]$$

If  V is a $k$-vector space we write $V^*$ for its linear dual.

If  $L$ is a perfect field, we write $G_L$ for its absolute Galois group (for a choice of an algebraic closure).

\vskip2mm

\subsection{Ordinarity}\label{ss:ord}

Let  $K$ be a $p$-adic field, $k$ its residue field, and $A$ a local $W(k)$-algebra. A representation $\rho: G_K \rightarrow  \GL_{2}(A)$ is called \textit{ordinary of weight two} if it has the form

\begin{equation}\label{ordinary_def}  \left( \begin{array}{cc} \omega  \varepsilon & * \\ 0 & \varepsilon^{-1} \end{array} \right)\end{equation}
where $\varepsilon : G_K \rightarrow A^{\times}$ is unramified, $\bar{\varepsilon}^2 \neq 1$ for $\ov{\varepsilon}:=\varepsilon \mod \fm_{A}$, and
$$\omega: G_K \rightarrow \mathbb{Z}_p^{\times} \rightarrow A^{\times}$$
is the cyclotomic character. (Actually, $\bar{\varepsilon}^{2}\neq 1$ is an additional hypothesis, often referred to as the $p$-distinguished hypothesis.)

We will write $A[\chi]$ for the free $A$-module of rank 1 on which $G_K$ acts by the character $\chi$. The coefficient $*$ defines a class $e \in \mathrm{Ext}_K^1(A[\varepsilon^{-1}], A[\omega \varepsilon]) = H^1(K, A[\omega \varepsilon^2]).$  

For a global field $F$, a representation $\rho$ of the Galois group into $\GL_{2}(A)$ is called \textit{ordinary of weight two} if its restriction to $F_v$ (for any prime $v$ above $p$) is ordinary of weight two. We also assume that the determinant of $\rho$ is the cyclotomic character. 

We will consider representations of $\Gamma_n$, thus unramified outside $S$.
 For the places in $S$ away from $p$, we impose no conditions (`unrestricted deformations'.) (We could impose local conditions, given by compatible deformation data $(\mathcal{D}_{n,\mathfrak{q}})$ for the primes $\mathfrak{q}$ dividing $ S\setminus \{\fp|p\}$, the conditions being compatible with respect to the field extensions. However it seems delicate to check the arguments of \S\ref{s:WL_ad} in this more general situation.)

\vskip2mm

\vskip2mm
Let $k$ be a finite field of characteristic $p$. 
Let $\bar{\rho}: \Gamma_0 \rightarrow \GL_{2}(k)$
be an absolutely irreducible representation satisfying the following. 
\begin{itemize}
\item[(ord)] $\bar{\rho}$ is ordinary of weight 2.
\item[(irr$_{F(\zeta_{p})}$)] $\bar{\rho}|_{G_{F(\zeta_{p})}}$ is irreducible. 
\item[(NS)] The restriction of $\bar{\rho}$ to $ F_{\mathfrak{p}}$ is absolutely indecomposable\footnote{See Remark \ref{Ex} (3) for the general case.} for all $ \mathfrak{p}$.
\item[(det)] The determinant is the cyclotomic character. 
\end{itemize}
\vskip2mm

\vskip2mm
In particular  $\bar{\rho}$ is totally odd (the image of each complex conjugation has determinant $-1$).

Note that these conditions remain satisfied when  $\bar{\rho}$ is restricted to $F_n$: $\varepsilon^2$ remains non-trivial as 
$F_{n, \mathfrak{p}} / F_ {\mathfrak{p}}$ is totally ramified, and then inflation-restriction implies that $H^1(K, k[\omega \varepsilon^2]) \rightarrow H^1(K_n, k[\omega\varepsilon^2])$ is injective ($K=F_{\mathfrak{p}} \subset K_n= F_{n, \mathfrak{p}}$). In particular, for all $n$, $\bar{\rho} \vert _{G_{K_n}}$ is indecomposable. The same argument applies to the restriction to  $G_{K_n(\zeta_p)}$.Thus $\bar{\rho}$, restricted to $G_{K_n(\zeta_p)}$, is semi-simple by Clifford theory (\cite[Thm. 1.1]{CR}) and indecomposable, and therefore irreducible.
In this paragraph and henceforth, we let $\varepsilon=\ov{\varepsilon}_\fp$, the latter as in \eqref{ordinary_def} for $\ov{\rho}$ and $\omega$ also denotes the mod $p$ cyclotomic character of $G_{F_\fp}$.

Write $\widehat{\mathcal{C}}_W$  for the category of complete local $W$-rings ($W=W(k)$) with residue field $k$; write ${\mathcal{C}}_W$ for the subcategory of \textit{Artinian} objects in $\widehat{\mathcal{C}}_W$. (Cf.~\cite[p.~267]{Ma}. Note however that  we do not assume rings in $\widehat{\mathcal{C}}_W$ to be Noetherian.) We simply write $\Hom(-,-)$ for the \textit{continuous} homomorphisms in  $\widehat{\mathcal{C}}_W$.  For the representability properties it suffices to consider liftings of $\bar{\rho}$ to elements of ${\mathcal{C}}_W$.

For any non-negative integer $n$, there exists a universal deformation ring $R_n$ over $W(k)$, the \textit{ordinary deformation ring} for $F_n$ parametrising ordinary liftings (of weight 2) of $\bar{\rho}$ over algebras in  ${\mathcal{C}}_W$. By results which are now well-known, we have

\begin{thm}\label{CN}
$R_n$ is a complete Noetherian algebra in $\widehat{\mathcal{C}}_W$ for finite $n$. 
\end{thm}

\subsection{Deformation rings over $F_{\infty}$}

\vskip2mm

By construction, for $A \in \widehat{\mathcal{C}}_{W}$, there exists a natural bijection  
$$\Hom(R_n,A) \leftrightarrow \rho_A^n = \{\text{ordinary deformations of $\ov{\rho}|_{\Gamma_{n}}$ over $A$}\}$$
(the representations on the right taken modulo conjugation by $1+ \mathfrak{m}_A M_2(A)$). 

By restriction $\rho_A^n$ yields an ordinary representation for $\Gamma_{n+1}$. Taking $A=R_n$ we see that there exists a natural homomorphism $R_{n+1} \rightarrow R_n$.

\begin{lem}\label{sur}
The homomorphism $R_{n+1} \rightarrow R_n$ is surjective.
\end{lem}
\begin{proof}
We have the tangent spaces
$$(\mathfrak{m}_R/(p, \mathfrak{m}_{R}^2))^{*} = H^1_{\ord}(\Gamma, \Ad^0 \bar{\rho}) $$
where $\Ad^0 \bar{\rho}$ is the representation of $\Gamma$ on the traceless endomorphisms of the space of $\bar{\rho}$ (see \S\ref{s:local})
for $\Gamma = \Gamma_n, \Gamma_{n+1}$, and $R= R_n, R_{n+1}$
 (cf.~\cite{CHT}). 
 The definition of $H^1_{\ord}$ is recalled in \S\ref{ss:tg}.

\vskip2mm

  Note that  $F_{n+1,S} = F_{n,S}$ . Consider the exact sequence
 $$1 \rightarrow \Gamma_{n+1} \rightarrow \Gamma_n \rightarrow \Delta_{n,n+1} \rightarrow 1$$
  where $\Delta_{n,n+1} = \Gal(F_{n+1}/F_n)$. This yields the exact sequence
$$0 \rightarrow H^1(\Delta, H^0(\Gamma_{n+1},W)) \rightarrow H^1(\Gamma_n,W) \rightarrow H^0(\Delta, H^1(\Gamma_{n+1},W)).$$

Since the representation of $\Gamma_{n+1}$ on $V$ is indecomposable, $H^0(\Gamma_{n+1},W)=0$, whence an exact sequence
\begin{equation}
0 \rightarrow H^1(\Gamma_n, \Ad^0 \bar{\rho}) \rightarrow H^1(\Gamma_{n+1}, \Ad^0 \bar{\rho} )  
\end{equation}

\vskip2mm

Now the definition of ordinary cohomology (see \S\ref{ss:tg}) yields a commutative diagram

\vskip6mm
$\begin{tikzcd}[row sep=2.5em]
& 
H^1_{\ord}(\Gamma_n,W)  \arrow[r,] \arrow[d,]&
H^1_{\ord}(\Gamma_{n+1},W)           \arrow[d,]\\
0 \arrow[r,]&
H^1(\Gamma_n,W) \arrow[r,]&
H^1(\Gamma_{n+1},W)\\
\end{tikzcd}$

(the local conditions defining $H^1_{\ord}$ being compatible), with injective vertical maps, whence

\begin{equation}
0 \rightarrow H^1_{\ord}(\Gamma_n, \Ad^0 \bar{\rho}) \rightarrow H^1_{\ord}(\Gamma_{n+1}, \Ad^0 \bar{\rho} )  
\end{equation}

This yields first $R_{n+1} \otimes k \twoheadrightarrow R_n \otimes k$ since these algebras are Noetherian and complete, and then  $R_{n+1} \twoheadrightarrow R_n  $ as both algebras are $p$-complete.
\end{proof}
\vskip2mm

Now we define $$R_\infty=  \varprojlim R_n.$$ It belongs to $\widehat{\mathcal{C}}_W$. It is \textit{not} known to be Noetherian. (Compare \cite[pp. 354-357]{Hi}.)

We now  want to consider ordinary deformations of $\bar{\rho}\vert_ {F_\infty}$. First note that $\bar{\rho}\vert_{\Gal(F_{\infty,S}/F_{\infty})}$ remains ordinary of weight $2$ (with the previous definitions); in particular $\varepsilon^2 \neq 1$ on this subgroup. The exact sequence
$$1 \rightarrow \Gal(\bar{K}/K_\infty) \rightarrow \Gal(\bar{K}/K) \rightarrow \Delta \rightarrow 1$$
where $\Delta \cong \BZ_p$, yields again
$$0 \rightarrow H^1(\Delta, H^0(\bar{K}/K_\infty, k)) \rightarrow H^1(\bar{K}/K, k) \rightarrow H^0(\Delta, H^1(\bar{K}/K_\infty, k))$$
where $k$ is endowed with the representation $\omega \varepsilon^2$, so the class of $e$ in $H^1(\bar{K}/K_\infty,k)$ is non-zero as the first term vanishes ($\omega \varepsilon^2$ being equal to $\varepsilon^2$ on the subgroup).

\vskip2mm

However standard deformation theory does not seem to apply here. Indeed:
\begin{itemize}
\item[(i)] The group $\Pi= \Gal(F_S/F_\infty)$ does not satisfy the usual finiteness condition, viz., $\Hom(\Pi, \BZ/p \BZ)$ being finite. In fact all we seem to know is that $\Pi^{\ab}$ is finitely generated over the $\BZ_p$-Iwasawa algebra $\Lambda$ (Cf. \cite[p. 735]{NSW}).
\item[(ii)] Even with a proper definition of $H^1_{\ord}(\Pi, \Ad^0(\bar{\rho}))$, this may not be finite without further conditions.
\end{itemize}

Nevertheless we will see that $R_\infty$ still represents the natural deformation problem. (See also Dickinson's appendix to \cite{Go}.) We first have:

\begin{lem}\label{lim}
For $A \in {\mathcal{C}}_W$,
$$\Hom(R_\infty,A) = \varinjlim \Hom(R_n,A).$$
\end{lem}
\begin{proof}
This is clear since $A$ is finite and $R_\infty$ is the projective limit of compact rings. Note that $\Hom(R_n,A) \subset \Hom(R_{n+1},A)$.
\end{proof}

\begin{prop}\label{descent}
Let $(A, \rho_A)$ be an ordinary deformation of $\bar{\rho}\vert_\Pi$ to $A \in {\mathcal{C}}_W$. Then there exists $n < \infty$ such that $\rho_A$ extends to $\Gal(F_S/F_n)$.
\end{prop}

(By `ordinary' we mean henceforth verifying the condition \eqref{ordinary_def}.)

\begin{proof}
 As before we have an exact sequence
$$1 \rightarrow \Pi \rightarrow \Gal(F_S/F) \rightarrow \Delta \rightarrow 1$$
with $\Delta \cong \BZ_p$. 
The choice of a lifting of  a topological generator of $ \Delta$ gives a splitting; we identify $\Delta$ with its image by this section. 

Now $\Delta$ acts continuously on $\Pi$ by conjugation. Let $\Pi_1 \subset \Pi$ be the kernel of $\rho_A$, an invariant subgroup of finite index. There exists a subgroup of finite index $\Delta_1 \subset \Delta$ such that $$ \delta g \delta^{-1} \equiv g \mod {\Pi_1}$$ for $\delta \in \Delta_1.$ 

We can then set $\rho_A(g \delta)= \rho_A(g)$ for $g\in \Pi, \delta \in \Delta_1$; $\Delta_1$ corresponds to a finite extension $F_n$ and $\rho_A$ extends to $\Gal(F_S/F_n)$ (cf.~\cite[\S3.3]{Cl2}).  

 This yields a representation of $\Gamma_n$, but it is not yet ordinary. However the lower left coefficient of the matrix is a continuous function with values in A, vanishing on $\Pi$. Thus it vanishes on $\Gamma_{n'}$ for some $n' \geq n$. Likewise, the diagonal will be given by $(\omega \varepsilon, \varepsilon^{-1})$ upon restriction to $\Gamma_{n''}$, since $A$ is finite.
Similarly, one  checks that the deformation of this extension (rather than the lifting) is well-defined. 
\end{proof}
\begin{cor}\label{univ}
$R_\infty$ represents the ordinary deformations of $\bar{\rho}\vert_\Pi$.
\end{cor}
Note in particular that there is a natural universal deformation of $\bar{\rho}\vert_\Pi$, over $R_\infty$, defined by $ \varprojlim  \rho_n$.

\subsection{Main result}\label{ss:mr}The purpose of this paper is the following theorem. 
\begin{thm}\label{main}
Let $\bar{\rho}: G_F \rightarrow \GL_{2}(k)$
be an absolutely irreducible representation as in \S\ref{ss:ord}.
 Let $(\rho,V)$ be a deformation of $\bar{\rho}$ over the integer ring of a $p$-adic field, $V$ the underlying vector space and let $T \subset \Ad^0V$ be a $G_F$-stable lattice.
Assume $R_\infty$ is Noetherian.
   Assume further that 
   \begin{itemize}
  \item[(Aut)] $\rho$ is automorphic,   
  \item[(ad$_{F(\zeta_{p})}$)] $\bar{\rho}|_{G_{F(\zeta_{p})}}$ is adequate and 
  \item[($\mu$)] $\mu(X^{1}(F,T^{*}(1))_{\tor})=0=\mu(X^{1}(F,T)_{\tor})$ 
   \end{itemize} 

Then it is formally smooth, i.e.
$$R_\infty \simeq W(k)[\![X_1,...,X_s]\!]$$
for some $s \geq 1$.
\end{thm}

\vskip2mm

(Refer to \S\ref{s:WL_ad} for the definition of  the Iwasawa modules $X^1$ and the corresponding $\mu$-invariants, and the notion of `adequate'.)

\vskip2mm

\begin{remark}\label{Ex}
\noindent
\begin{itemize}
\item[(1)] For conditions on the data ensuring that $R_\infty$ is Noetherian, see \cite[Cor. 5.11]{Hi}.
\item[(2)] Let $F$ be $\BQ$ and 
$$p\in\{11,17,19,23,29,31,37,41,43,47,59,61,67\}.$$
Then there exists a $p$-ordinary $f\in S_{2}(\Gamma_{0}(p))$ such that $R_{\infty} \simeq W(k)[\![X]\!]$ (cf. \cite[Ex. 1.68]{Hi}). Here we consider deformations of the associated mod $p$ Galois representation.
\item[(3)] The hypothesis ({\rm NS}) is inessential. It is currently used for arguments in section \ref{s:local}, specifically Lemma \ref{spld}, which is key for the proof of Theorem \ref{main}. However, even otherwise, the lemma remains true. (Basically, various exact sequences in section \ref{s:local} are split otherwise and can be analysed directly.) The details will appear in \cite{BCM}. 
\end{itemize} 
\end{remark}
\begin{remark} In light of Hida theory, one has $\dim(R_{\infty})\geq 2$. An outline: 

First, assume $F=\BQ$ and $F_\infty= \BQ_\infty$ the $\BZ_p$-extension of $\BQ$. 
In this case the $R_{n}$'s are finite over $\BZ_{p}$ but $R_{\infty}$ has Krull dimension at least two: 
let $f$ be a weight 2 eigenform, ordinary. Then there is a Hida family $\mathcal{F}$ through $f$ (cf.~\cite[Cor. 3.2]{Hi'}), whence
$$\tilde{\rho} : G_{\BQ,S} \rightarrow \GL_{2}(\BI)$$
for $\BI$ finite over $\BZ_p[\![X]\!]$ (cf.~\cite[Thm. II]{Hi0}). In Hida's construction, $\tilde{\rho}$ parametrises a family of representations of varying weights. However:
\begin{lem}
$\tilde{\rho}\vert _{\Gal(\BQ_{\infty, S}/ \BQ_\infty)}$ is of weight 2.
\end{lem}
\begin{proof}
For example, suppose that the Hida family $\mathcal{F}$ corresponds to $(f_k)$ where $f_k$
has weight $k$, and $f_2= f$.  Set
$$n_r = 2 + (p-1) p^r.$$  
Then the base change of $f_{n_r}$ to $F_r$, reduced modulo $p^r$, has weight $2$.  
\end{proof}
Thus the surjection $R_\infty \rightarrow \BI$ yields  
 $\dim(R_\infty) \geq 2$
without the hypotheses in Theorem \ref{main}. (The surjectivity just follows by considering the traces of Frobenii for the universal representation.) A similar argument applies to the general case (cf. \cite[Thm. II]{Hi1}).
\end{remark}

\section{Local cohomology}\label{s:local}

In this section $K=F_{n,\mathfrak{p}}$ is local and $\bar{\rho}$ is an ordinary representation of $G_K$ (verifying the conditions of \S\ref{s:Intro}). In particular the extension class arising from $\bar{\rho}$ is non-split (cf.~({\rm NS})). For simplicity we write $d$ for $d_{\mathfrak{p}} = [F_{\mathfrak{p}}:\BQ_p].$ 

Let $V=k^2$ be the space of $\bar{\rho}$, and $W = \End^0(V)$ be the space of traceless endomorphisms of $V$. It is endowed with the natural representation $\Ad^0(\bar{\rho})$. Let $\Ad^0\bar{\rho}(1)$  be the Tate twist, the tensor product  $\Ad^0(\bar{\rho}) \otimes k[\omega]$ with the cyclotomic character. (Recall that $k[\chi]$ is the module associated to a character $\chi$; $V(1)= V \otimes k[\omega].)$ The main result is Lemma \ref{spld}. 

\subsection{Local cohomology of the adjoint}
Let  $W_0 \subset W_1 \subset W_2=W$ be the filtration of $W$:
\begin{equation}\label{udl}
W_{0}=\bigg{\{}\left( \begin{array}{cc} 0 & * \\ 0 & 0 \end{array}\right)\ \bigg{\}}, 
\qquad 
W_{1}=\bigg{\{}\left( \begin{array}{cc} * & * \\ 0 & * \end{array} \right)\ \bigg{\}}
\end{equation}
preserved by $G_K$. Then as $G_K$-modules,
$$W_0  \cong k[\varepsilon^2 \omega], \qquad W_1/W_0 \cong k[\mathds{1}], \qquad W_2/W_1 \cong k[\omega^{-1} \varepsilon^{-2}] $$ 
for $\mathds{1}$ being the trivial character. 

The exact sequence
$$0 \rightarrow W_0 \rightarrow W_1 \rightarrow W_1/W_0 \rightarrow 0$$
induces
\vskip2mm
$H^0(K, W_1) \rightarrow H^0(K, W_1/W_0) \rightarrow H^1(K, W_0) \rightarrow H^1(K, W_1) \rightarrow H^1(K, W_1/W_0) \rightarrow H^2(K, W_0)  \rightarrow H^2(K, W_1)  \rightarrow H^2(K, W_1/W_0) \rightarrow 0.$
\vskip2mm
Write $h^i(K, -) = \dim_{k}H^i(K,-).$

\begin{lem}\label{dim}

\item[(i)] $h^0(K,W_1/W_0)= 1$ and the map $H^0(K, W_1/W_0) \rightarrow H^1(K, W_0)$ is injective.
\item[(ii)] $h^1(K,W_0)= p^n d$.
\item[(iii)] $h^1(K, W_1) = 2 p^n d.$
\item[(iv)] $h^1(K, W_1/W_0) = p^n d +1$.
 \item[(v)] $ h^2(K,W_1)=  0 $.

 \end{lem}
 \begin{proof}
 Write $V'= V^*(1)$. Then $$W_0' \cong k[\varepsilon^{-2}], \qquad (W_1/W_0)' \cong k[\omega].$$ By Tate duality we see that $h^2(K, W_0) = h^2(K, W_1/W_0)= 0 .$ This implies (v). 
 
 The first part of  (i) is obvious; we have $h^0(K,W_1)=0$ since the extension is non-split, so the map is injective. The map $H^1(K, W_1) \rightarrow H^1(K, W_1/W_0)$ is surjective since $H^2(K,W_0)=0$.  Now the formulas (ii)-(iii) follow from Tate's Euler-Poincar\'{e} formula and (iv) from the exact sequence. 
 \end{proof}
 
 Now recall that for $X$ a representation of $G_K$ on a $k$-vector space, 
 $$H^1_{\mathrm{nr}}(K,X) = \ker \{H^1(G_K,X) \rightarrow H^1(I_K,X)\}$$
where $I_K$ is the inertia. We \textit{define} the unramified classes $H^1_{\ur}(K, W_1)$ to be the inverse image of $H^1_{\mathrm{nr}}(K, W_1/W_0).$ 
  
  At this point we have the exact sequence
 \begin{equation}\label{udses}
 0 \rightarrow H^0(K, W_1/W_0) \rightarrow H^1(K, W_0)   \rightarrow H^1(K, W_1)    \rightarrow H^1(K, W_1/ W_0) \rightarrow  0 
\end{equation}
where the corresponding dimensions are $ (1, p^n d, 2 p^nd, p^n d+1)$. Since $W_1/W_0$ is with trivial $G_K$-action, $H^1_{\rm{nr}}(K, W_1/W_0) \cong k$. Thus

$$
\dim_{k}H^1_{\ur}(K, W_1)= p^n d.
$$

 Now the exact sequence
 $$0 \rightarrow W_1 \rightarrow W \rightarrow W/W_1 \rightarrow  0 $$
 induces
 \begin{equation}\label{dlses}
  0 \rightarrow H^1(K, W_1)  \rightarrow H^1(K,W) \rightarrow H^1(K, W/W_1) \rightarrow 0
 \end{equation}
 by Lemma \ref{dim}. 
   
   We \textit{define} $H^1_{\ord}(K, \Ad^0 \bar{\rho})$ as the image of $H^1_{\ur}(K, W_1)$ in $H^1(K,W)$. We also note the vanishing of $H^2(K, \Ad^0 \bar{\rho})$ by the analogue of (\ref{dlses}) for $H^2$, and Tate duality for $W/W_1$.
   
   We summarise the results obtained so far:
   \begin{lem}\label{lc1}
   \item[(i)] $H^0(K, \Ad^0 \bar{\rho}) = H^2(K, \Ad^0 \bar{\rho}) = 0.$
   \item[(ii)] $\dim_{k} H^1_{\ord}(K, \Ad^0 \bar{\rho}) = p^n d$.
   \item[(iii)] $\dim_{k} H^1(K, \Ad^0 \bar{\rho}) = 3 p^n d.$
 
   \end{lem}
   
   (The third equality coming from (i) and the Euler-Poincar\'{e} formula applied to $W$.)
   
   Now consider the extension $K=F_{n, \mathfrak{p}} = K_n$ of $K_0= F_{\mathfrak{p}}$, whence an action of $\Delta_n= \Gal(K_n/K_0)$ on the cohomology groups $H^*(K_n, -).$
\begin{lem}\label{sec:freeness}
$H^1(K_n, \Ad^0\bar{\rho} (1))$ is free over $k[\Delta_n]$ of rank $3d$ .
\end{lem}
\begin{proof}   
   Write $M_{n}= H^1(K_n, \Ad^0 \bar{\rho}(1)).$ Note that $W$ is self-dual, so $\dim_{k} M_{n} = 3 p^n d$ by Lemma \ref{lc1} and Tate duality. 
   
   We show that the space of coinvariants $H_0(\Delta_n,M_{n})$ has dimension $3d$: this implies by Nakayama's lemma that there is a surjective map $k[\Delta_n]^{3d} \rightarrow M_{n}$, and we conclude by counting dimensions. 
   
   However, the dual of $H_0(\Delta_n, M_{n})$ is $H^0(\Delta_n, H^1(K_n, \Ad^0\bar{\rho}))$; this is isomorphic to $H^1(K_0, \Ad^0\bar{\rho})$ by inflation-restriction as $H^0(K_n, \Ad^0\bar{\rho}) = 0$. By Lemma \ref{lc1}, the dimension of this space is $3d$.
 \end{proof}  
   \vskip2mm
   
   We now consider the subspace $H^1_{\ord}(K_n, \Ad^0 \bar{\rho})$, of dimension $p^n d$. Note that the filtration $W_i$ of $W$ gives rise to cohomology  spaces on which $\Delta_n$ acts.
   
   \begin{lem}\label{inv}
   
  $H^1_{\mathrm{nr}}(K_n, W_1/W_0)$, $H^1_{\ur}(K_n, W_1)$ and $H^1_{\ord}(K_n, \Ad^0\bar{\rho})$ are invariant by the action of $\Delta_n$.
   
   \end{lem}
   \begin{proof}
   It suffices to check this for the first space, and this is obvious as the inertia $I_n$ is invariant by $\Delta_n$.
   
   \end{proof}
   
   In $k[\Delta_n]$, the space of $\Delta_n$-invariants is $$\bigg{\lbrace} f= x \cdot \sum_{\Delta_n} \delta \big{|} x \in k \bigg{\rbrace} .$$ 
   
   The space $H^0(\Delta_n, k[\Delta_n]^d)$ is the sum of these lines. If $j_1, j_2$ are two injections of the trivial $\Delta_n$-module into $H^0(\Delta_n,k[\Delta_n]^d)$, it follows that there is a $\Delta_n$-equivariant isomorphism of $k[\Delta_n]^d$ conjugating them. We write $k[\Delta_n]^d /k$ for the quotient, independent of the map up to isomorphism as a $k[\Delta_n]$-module.
 
 \begin{lem}\label{spl}
 $H^1_{\ord}(K_n, \Ad^0 \bar{\rho})$ is isomorphic, as a $k[\Delta_n]$-module, to  $$(k[\Delta_n]^d /k) \oplus k$$ with $k$ being the trivial $k[\Delta_n]$-module.
  \end{lem}
 \begin{proof}  
    Indeed the exact sequence (\ref{udses}) yields first
    \begin{equation}\label{urses}
    0 \rightarrow H^1(K_n, W_0)/\Im H^0(K_n, W_1/W_0) \rightarrow H^1_{\ur}(K_n, W_1) \rightarrow H^1_{\mathrm{nr}}(K_n, W_1/W_0) \rightarrow 0
   \end{equation}
 with $H^1_{\ur}(K_n, W_1) \cong H^1_{\ord}(K_n, \Ad^0\bar{\rho})$ and the dimensions being $(p^nd-1, p^n d, 1)$. The argument given for Lemma \ref{sec:freeness} shows that $H^1(K_n, W_0)$ is free of rank $d$ over $k[\Delta_n]$. 
     
     Recall that $W_1/W_0$ is the trivial module (for $G_{K_0}$). It follows that $$H^1_{\mathrm{nr}}(K_n, W_1/W_0) = \Hom (U,  W_1/W_0),$$ where $U= \Gal(K_n^{\mathrm{nr}}/K_n)= \Gal(K_0^{\mathrm{nr}}/K_0)$, is the trivial module for $k[\Delta_n]$. Similarly,  the image of $H^0(K_n, W_1/W_0) \cong k$ is trivial. 
     
     Finally, the exact sequence is split: by the previous argument computing $H^1_{\mathrm{nr}}(K_n, W_1/W_0)$ we can fix an element $\alpha \in H^1_{\mathrm{nr}}(K_0, W_1/W_0)$ that is a basis of $H^1_{\mathrm{nr}}(K_n, W_1/W_0)$. We then lift it to $\beta \in H^1_{\ur}(K_0, W_1)$: its restriction to $K_n$ is an element $\beta_n \in H^1_{\ur}(K_n, W_1)$ that is $\Delta_n$-invariant. 
     
\end{proof}

 Consider now $\pi: \Delta_{n+1} \twoheadrightarrow \Delta_n$. This induces a natural map $k[\Delta_n] \hookrightarrow k[\Delta_{n+1}]$, $f(\delta) \mapsto f(\pi \delta),$ dual to the projection of Iwasawa theory. It is equivariant under the action of $\Delta_{n+1}$, acting on $k[\Delta_n]$ via the quotient map.
 
 \begin{lem}\label{ocp}
 The restriction $H^1_{\ord}(K_n, \Ad^0 \bar{\rho}) \rightarrow H^1_{\ord}(K_{n+1}, \Ad^0 \bar{\rho})$ is injective. It is compatible with the splitting of Lemma \ref{spl}, and equivariant for the action of $\Delta_{n+1}$. 
 \end{lem}
 \begin{proof}
 Write $0 \rightarrow H'_n \rightarrow H_n \rightarrow L_n \rightarrow 0$ for the exact sequence  (\ref{urses}), with $H_n = H^1_{\ur}(K_n, W_1) \cong H^1_{\ord}(K_n, \Ad^0\bar{\rho})$. We get natural maps

   \begin{diagram}
 0 &\rightarrow &H'_n  &\rightarrow & H_{n} &\rightarrow &L_{n} &\rightarrow &0\\
 & & \big{\downarrow} & &\big{\downarrow} & &\big{\downarrow}\\
0 &\rightarrow &H'_{n+1}  &\rightarrow & H_{n+1} &\rightarrow &L_{n+1} &\rightarrow &0\\
\end{diagram}

   As shown in the proof of Lemma \ref{spl}, $L_n= \Hom(U, W_1/W_0) = L_{n+1}$ since $\Gal (K_n^{\mathrm{nr}}/ K_n)= \Gal (K_{n+1}^{\mathrm{nr}}/ K_{n+1})$. We are reduced to looking at the map $\Res: H^1(K_n, W_0) \rightarrow H^1(K_{n+1}, W_0)$. Both spaces contain the line $k = \Im(H^0)$, on which restriction is an isomorphism. Finally, 
   
   $$H^1(K_n, W_0) \cong k[\Delta_n]^d, \qquad H^1(K_{n+1}, W_0) \cong k[\Delta_{n+1}]^d$$ 
by the exact analogue of Lemma \ref{lc1}. The two isomorphisms are respectively as modules over $\Delta_n$ and $\Delta_{n+1}$. As $W_0 = k[\varepsilon^2 \omega],$ $H^1(K_n, W_0) \rightarrow H^1(K_{n+1}, W_0)$ is injective. This proves the first part of the lemma.
   
   In fact we can be more precise. As in the proof of Lemma \ref{lc1}, $H^1(K_n, W_0) \cong k[\Delta_n]^d$ was deduced, through Nakayama's lemma, from $$H^0(\Delta_n, H^1(K_n, W_0^*(1))) \cong H^1(K_0, W_0^*(1)) \cong k^d,$$ dual to $H^1(K_0, W_0) \cong k^d$. The last isomorphism is independent of $n$. As a consequence, the restriction $H^1_{\ord} (K_n, \Ad^0(\bar{\rho})) \rightarrow H^1_{\ord} (K_{n+1}, \Ad^0(\bar{\rho}))$ is given (on the spaces $H'_n)$, in a suitable basis of the free modules, by taking the natural map
   $$k[\Delta_n]^d  \hookrightarrow  k[\Delta_{n+1}]^d$$
and quotienting through a line $\sum_1^d x_i \sum_{\Delta_n} \delta$, sent to $\sum_1^d x_i \sum_{\Delta_{n+1}} \delta$  $(x_i \in k).$ 
 
 The other assertions of the lemma are now clear.
 
 \end{proof}
 
 \subsection{Local cohomology, dualised}
 
 \vskip2mm
 
 We now use the Tate pairing
  $$ H^1(K_n, \Ad^0\bar{\rho}) \times H^1(K_n, \Ad^0\bar{\rho}(1)) \rightarrow k.$$

   Let $ H^1_{\ord, \bot} \subset H^1(K_n, \Ad^0\bar{\rho}(1))$ be the orthogonal space of $H^1_{\ord}$. We set
$$ H^1_{\ord, *}(K_n, \Ad^0\bar{\rho}(1)) =  H^1(K_n, \Ad^0\bar{\rho}(1))/ H^1_{\ord,\bot}.$$
      So this is naturally dual to $H^1_{\ord}$. When $K_n$ is concerned, we write $H^1_{\ord , n}$ etc. We can take the limit of these spaces under corestriction. In fact we obtain naturally a diagram
\begin{diagram}
 H^1_{\ord, *, n+1} &\cong &(H^1_{\ord, n+1})^*\\
\big{ \downarrow} & &\big{\downarrow}\\
 H^1_{\ord, *, n} &\cong  &(H^1_{\ord, n})^*\\
\end{diagram}
 where the surjection on the right comes from the previous injection (Lemma \ref{ocp}) and the surjection on the left completes the diagram. 
 We must however check that this is given by corestriction on the left: i.e., that for $\beta \in H^1_{\ord, *, n+1}$ and $\alpha \in H^1_{\ord,n}$, 
 $$
 ( \Cor ~ \beta, \alpha ) = ( \beta, \Res~ \alpha ) \in (1/p)\BZ/\BZ = \BF_p.
 $$ 
 (We assume $k= \BF_p$; in  general an easy argument of restriction of scalars reduces to this case.) 
 
 The duality is given by the cup-product, with values in $H^2(K, \mu_{p^\infty}) [p] = (1/p)\BZ/\BZ = \BF_p. $ The general formula is $\Cor(\beta \cup \Res \alpha) = \Cor \beta \cup \alpha.$ For the canonical identification of $\Br(K)$ with $\BQ/ \BZ$, the restriction $\Br(K_n) \rightarrow \Br(K_{n+1})$ is given by $\alpha \mapsto p \alpha$ (cf.~\cite[XIII,~\S3]{Se}); on the other hand $\Cor \circ \Res : \Br(K_n) \rightarrow \Br(K_{n})$ is also  $\alpha \mapsto p  \alpha$. Thus $\Cor(p\alpha) = p \alpha$ for $\alpha \in \Br(K_{n+1})$ and $\Cor: H^2(K_{n+1}, \mu_{p^\infty}) [p]\rightarrow H^2(K_n, \mu_{p^\infty}) [p]$ is bijective\footnote{This is certainly well-known but we could not find a reference.}.
 
 \vskip2mm
 
 We now dualise the expression of $H^1_{\ord}(K_n, \Ad^0\bar{\rho})$ obtained in Lemma \ref{spl}. As in the proof of Lemma \ref{sec:freeness}, write $M_n$ for $H^1(K_n, \Ad^0\bar{\rho}(1))$ 
 and $M_n^0$ for $M_n/k$. Thus
 
$$H^1_{\ord}(K_n, \Ad^0\bar{\rho})^* \cong (M_n^0 \oplus k)^* \cong  (M_n^0)^* \oplus k,$$
 and
 $$0 \rightarrow E_n \rightarrow  k[\Delta_n]^d \rightarrow M_n^0 \rightarrow 0 $$ 
where $E_n \cong k$.

If we restrict to $K_{n+1}$, the corresponding map $k \rightarrow k$ is an isomorphism as was seen in the proof of Lemma \ref{spl}. We can now choose the line $E_n$ equal to $(e_n, 0,...,0) \in k[\Delta_n]^d$ with $e_n= \sum_{\delta \in \Delta_n} \delta.$ Then $(k[\Delta_n]/E_n)^* = I_n$ is the augmentation ideal of $k[\Delta_n]$. We obtain
 
  $$\varprojlim H^1_{\ord,*,n} = \Omega^{d-1} \oplus \varprojlim I_n \oplus k.$$ 
The limit of the augmentation ideals is nothing but the augmentation ideal in $\Omega$: 
$$I= T \cdot k[\![T]\!] \subset k[\![T]\!] = \Omega.$$ Thus we have proved:
  
  \begin{lem}\label{spld} 
 As an $\Omega$-module,
     $$\varprojlim H^1_{\ord,*,n}(K_n, \Ad^0\bar{\rho}(1)) \cong \Omega^d \oplus k .$$
  \end{lem}
     
 \section{Ordinary global Galois cohomology}\label{s:ordinary}
 In this section we return to the global setup of ordinary deformation rings in \S\ref{s:Intro}.
\subsection{Tangent and obstruction space}\label{ss:tg}We will now compute, first for fixed $n$, the tangent and obstruction space of the ordinary deformation space for $\bar{\rho} \vert _{F_n}$. 

Note that we are looking at deformations with fixed determinant. The tangent and obstruction space are  then $H^1_{\ord}(\Gamma_n, \Ad^0(\bar{\rho}))$ and $H^2_{\ord}(\Gamma_n, \Ad^0(\bar{\rho}))$, which are given by the following exact sequence (see \cite[\S 2.2]{CHT}; recall that we are considering unrestricted deformations at the places in $S$ away from $p$):

     \vskip2mm
     \begin{equation}
\begin{aligned}  \label{lim}
0\rightarrow H^1_{\ord}(\Gamma_n, \Ad^0\bar{\rho}) \rightarrow  H^1(\Gamma_n, \Ad^0\bar{\rho}) \rightarrow \bigoplus _{\mathfrak{p}}  H^1(F_{n,\mathfrak{p}}, \Ad^0\bar{\rho})/  H^1_{\ord}(F_{n, \mathfrak{p}}, \Ad^0\bar{\rho}) \rightarrow 
H^2_{\ord}(\Gamma_n, \Ad^0\bar{\rho}) \rightarrow \\
H^2(\Gamma_n, \Ad^0\bar{\rho}) \rightarrow  \bigoplus _{\mathfrak{p}}  H^2(F_{n,\mathfrak{p}}, \Ad^0\bar{\rho}) \rightarrow H^3_{\ord}(\Gamma_n, \Ad^0\bar{\rho})   \rightarrow H^3(\Gamma_n, \Ad^0\bar{\rho}) \rightarrow 0 . 
\end{aligned}
\end{equation}
     
     \vskip2mm
     
 For the definition of $H^2_{\ord}$ and $H^3_{\ord}$ see \cite[Def.2.2.7]{CHT}. Note that restriction yields natural morphisms between these exact sequences relative to $F_n$ and $F_{n+1}$.

     \vskip2mm
     
     In particular, we obtain for the direct limits:
\begin{equation}   \label{limd}  
0 \rightarrow \varinjlim H^1_{\ord}(\Gamma_n) \rightarrow \varinjlim H^1(\Gamma_n) \rightarrow \bigoplus \varinjlim H^1(F_{n, \mathfrak{p}})/ H^1_{\ord}(F_{n, \mathfrak{p}}) \rightarrow \varinjlim H^2_{\ord}(\Gamma_n) \rightarrow ...
\end{equation}
where the coefficients are in $\Ad^0\bar{\rho}$. 
     
     The full cohomology spaces $H^i(\Gamma_n, \Ad^0\bar{\rho})$ can be fitted together by means of Shapiro's lemma:
     $$H^i(\Gamma_n, \Ad^0\bar{\rho})= H^i(\Gamma_0, \Ind_{\Gamma_n}^{\Gamma_0} \Ad^0 \bar{\rho})= H^i (\Gamma_0, \Ad^0\bar{\rho} \otimes k[\Delta_n])$$
since $\Ad^0\bar{\rho}$ extends to $\Gamma_0$. The group $\Gamma_0$ acts diagonally. 
     
     For $m\geq n$, the restriction map is then given by $k[\Delta_n] \hookrightarrow k[\Delta_m]$ (cf. before Lemma 2.6). Dually, the corestriction map : $H^i(\Gamma_m, \Ad^0\bar{\rho}) \rightarrow H^i(\Gamma_n, \Ad^0\bar{\rho})$ is then given by
     \begin{equation} \label{hi}
     H^i(\Gamma_0, \Ad^0\bar{\rho} \otimes k[\Delta_m])  \rightarrow H^i(\Gamma_0, \Ad^0\bar{\rho} \otimes k[\Delta_n])
    \end{equation}
(Cf. \cite[\S 6.3]{W}\footnote{Note that there the induced module is called coinduced.}) 
where $k[\Delta_m] \rightarrow k[\Delta_n] $ is the surjection defining the Iwasawa algebra.

     \subsection{Continuous Galois cohomology}\label{ss:ct} Before passing to the limit in (\ref{hi}), we must make some remarks on Galois cohomology. So far our Galois modules were discrete, and we were using the corresponding version of cohomology (cf.~\cite{Se1}). However (\ref{hi}) leads us to the limit

      $$\varprojlim k[\Delta_n] := \varprojlim \Omega_n =  \Omega,$$
seen as a $\Gamma_{0}$-module via $\Gamma \rightarrow \Delta.$ It is easy to see that this $\Gamma_{0}$-module is not discrete. On the other hand, if we endow $\Omega$ with its compact topology, $\Delta$ acts continuously.  We therefore consider the continuous cohomology $H^i_{\mathrm{ct}}(\Gamma_{0}, -)$ (cf.~\cite[II.7]{NSW}). 
     
     We now have, with $\Omega_n = k[\Delta_n]$:
     
     \begin{lem} \label{com}For all $i\geq 0$, there exists an exact sequence
     
     $$0 \rightarrow \sideset{}{^1}  \varprojlim H^{i-1}(\Gamma_{0}, \Ad^0\bar{\rho} \otimes \Omega_{n}) \rightarrow H^ i_{\mathrm{ct}}(\Gamma_{0}, \Ad^0 \bar{\rho} \otimes \Omega) \rightarrow  \varprojlim H^i (\Gamma_{0},\Ad^0 \bar{\rho} \otimes \Omega_n) \rightarrow 0 .$$
 (Cf. \cite[2.7.5 Theorem]{NSW}\footnote{This is a general result, cf \cite[p. 84]{W}.})      
     \end{lem}
     
     \vskip2mm
     
     In our case,  the groups of continuous cohomology are limits of finite-dimensional vector spaces, so the Mittag-Leffler condition is satisfied and $ \sideset{}{^1}  \varprojlim $ vanishes \cite [Ex. 3.5.2]{W}. In particular,

     \begin{equation}\label{eq1}
     \varprojlim H^1(\Gamma_0, \Ad^0\bar{\rho} \otimes \Omega_{n}) = H^1_{\mathrm{ct}}(\Gamma_0, \Ad^0\bar{\rho} \otimes \Omega). 
     \end{equation}

  \section{Weak Leopoldt for adjoint}\label{s:WL_ad} 
  In this section we consider the vanishing of the second global Galois cohomology for adjoint over the cyclotomic tower. 
  
  \subsection{Weak Leopoldt I} In this subsection we consider the vanishing of the second global Galois cohomology for adjoint with rational coefficients over the cyclotomic tower.   
  
    Let the notation and hypotheses be as in \S\ref{s:Intro}-\S\ref{s:ordinary}. Let $\rho$ be a deformation of $\bar{\rho}$ over the ring of integers $A$ of a $p$-adic field;  we also denote by $\rho$ the corresponding rational representation, on a space $V$.  Let $W$ denote $\Ad^{0}\rho(1)$ or $\Ad^{0}\rho$ and $T \subset W$ a Galois-stable lattice. 
    
  \begin{prop}\label{p-adic-WL-abs} Suppose that
  \begin{itemize}
  \item[(i)] $H^{0}(F_{\fp}, W^{*}(1))=0$ for $\fp|p$ and
  \item[(ii)] the localisation $H^{1}_{\mathrm{f}}(F,W^{*}(1)) \rightarrow \bigoplus_{\fp|p} H^{1}_{\mathrm{f}}(F_{\fp},W^{*}(1))$ is injective.
  \end{itemize}
  Then, 
  $$
  \varinjlim_{n}H^{2}(\Gamma_{n},W/T)=0
  $$
   \end{prop} 
  (See Perrin-Riou \cite[Prop. B.5]{PR}).
 
  \begin{remark} The above criteria for weak Leopoldt holds rather generally (cf.~\cite{PR}). 
  \end{remark}
 
  In view of Allen's result \cite[Thm. B]{Al}, we deduce the following. 
  \begin{cor}\label{p-adic-WL-aut} 
  Suppose that $W=\Ad^0\rho$ or $\Ad^0\rho(1)$ and 
   \begin{itemize}
  \item[(Aut)] $\rho$ is automorphic and 
  \item[(ad$_{F(\zeta_{p})}$)] $\bar{\rho}|_{G_{F(\zeta_{p})}}$ is adequate (\cite[Def. 3.1.1]{Al}).
 \end{itemize} 
  Then, 
  $$
  \varinjlim_{n}H^{2}(\Gamma_{n},W/T)=0.
  $$
 
  \end{cor}
  \begin{proof}
   The first hypothesis in Proposition \ref{p-adic-WL-abs} follows from our assumptions on $\bar{\rho}$ (\S\ref{ss:mr}).  
  
From \cite[Thm. B]{Al}, we have 
  $$
  H^{1}_{\mathrm{f}}(F,\Ad^{0}\rho^{*}(1))=0.
  $$
  
  We thus conclude 
   $$
  \varinjlim_{n}H^{2}(\Gamma_{n},\Ad^{0}\rho/T)=0.
  $$  
  As weak Leopoldt  (i.e., the conclusion of Proposition \ref{p-adic-WL-abs}) for a $p$-adic Galois representation $W$ implies the same for $W(j)$ with $j \in \BZ$ (\cite[1.3.3]{PR}), this finishes the proof. 
  \end{proof}
  \begin{remark} 
  \noindent
  \begin{itemize}
  \item[(1)] For $p > 5$, adequacy is equivalent to absolute irreducibility (\cite[Thm. A.9]{Th}).
  \item[(2)] The automorphy hypothesis (Aut) can be replaced with an analogous one involving potential automorphy (\cite[Thm.B]{Al}). Such a potential automorphy is indeed available under mild hypotheses (\cite[Thm. 4.5.2]{BLGGT}). 
  \end{itemize}

  \end{remark}

  \subsection{Weak Leopoldt II}\label{ss:WL-II} In this subsection we consider the vanishing of the second global Galois cohomology for adjoint with mod $p$ coefficients over the cyclotomic tower.  
  
  Let the notation and hypotheses be as in \S4.1. Let 
  $$
  X^{1}(F,T)= (\varinjlim_{n}H^{1}(\Gamma_{n},W^{*}(1)/T^{*}(1)))^{*},
  $$
  
  cf. \cite [1.3.1]{PR}. Recall that these groups are $\Lambda$-modules of finite type (\cite{PR}, ibid.)
  
   \begin{prop}\label{mod p-WL-abs} The following are equivalent. 
  \begin{itemize}
  \item[(i)] $ \varinjlim_{n}H^{2}(\Gamma_{n},p^{-1}T/T)=0
  $  
   \item[(ii)] $\varinjlim_{n}H^{2}(\Gamma_{n},W/T)=0$ and $\mu(X^{1}(F,T^{*}(1))_{\tor})=0$   
  for $\mu(\cdot)$ the Iwasawa $\mu$-invariant\footnote{See for example \cite[\S2]{Su}}.  
   \end{itemize}
(\cite[p. 126]{PR}).
  \end{prop} 
  
  \noindent 

\noindent
\begin{cor}\label{mod p-WL-aut} 
 
 Suppose that $W=\Ad^0\rho$ or $\Ad^0\rho(1)$ for an automorphic lift $\rho$ and $T \subset W$ is a stable lattice. Assume 
   \begin{itemize}   
  \item[(irr$_{F(\zeta_{p})}$)] $\bar{\rho}|_{G_{F(\zeta_{p})}}$ is irreducible and 
  \item[($\mu'$)] $\mu(X^{1}(F,T^{*}(1))_{\tor})=0$  
   \end{itemize} 
 
 Then, the dimensions

  $$
  \dim_{k} H^{2}(\Gamma_{n},p^{-1}T/T)
    $$
    are bounded as $n \rightarrow \infty$.
  \end{cor}
  \begin{proof} 
  It suffices to show that the dimensions
  $$
  \dim_{k} H^{1}(\Gamma_{n},W/T)/p, \\\ \dim_{k}H^{2}(\Gamma_{n},W/T)[p]
   $$
   are bounded as $n \rightarrow \infty$. 
  
  \begin{itemize}
  \item From \cite[(1.2) p. 10]{PR} and (irr$_{F(\zeta_{p})})$, 
  \begin{equation}\label{ctlH1}
  H^{1}(\Gamma_{n},W/T) \simeq \big{(}\varinjlim H^{1}(\Gamma_{m},W/T)\big{)}^{\Gal(F_{\infty}/F_{n})}.
  \end{equation}
  Note that the Pontryagin dual of $\big{(}\varinjlim H^{1}(\Gamma_{m},W/T)\big{)} ^{\Gal(F_{\infty}/F_{n})}  /p$ is the $\BZ_{p}$-submodule of $X^{1}(F,T^{*}(1))_{\Gal(F_{\infty}/F_{n})}$ annihilated by $p$ (\cite[p. 126]{PR}). 
  
  In view of structure theorem for finitely generated $\Lambda$-modules, 
  $$
   X^{1}(F,T^{*}(1))[p]_{\Gal(F_{\infty}/F_{n})} \sim
   X^{1}(F,T^{*}(1))_{\Gal(F_{\infty}/F_{n})}[p].    
    $$
  Here `$\sim$'  denotes up to bounded kernel and cokernel. 
  
  From hypothesis ($\mu'$), the $\Lambda$-module
  $X^{1}(F,T^{*}(1))[p]$ is trivial (\cite[p. 126]{PR}). 
  Thus, the $k$-modules 
  $X^{1}(F,T^{*}(1))_{\Gal(F_{\infty}/F_{n})}[p]$ are bounded (for example, \cite[Prop. 2.3.1]{Gr}).   
  
  We conclude that the dimensions $
  \dim_{k} H^{1}(\Gamma_{n},W/T)/p$ are bounded. 
   
 \item   
  From Corollary 4.3,  
  \begin{equation}\label{vanH2}
  \varinjlim_{n}H^{2}(\Gamma_{n},W/T)=0, 
  \end{equation}
  
  Thus, from \cite[(1.3) p. 10]{PR}
  \begin{equation}\label{ctlH2}
  H^{1}(\Gamma_{n}, \varinjlim H^{1}(\Gamma_{m},W/T)) \simeq H^{2}(\Gamma_{n},W/T). \end{equation}
 
  Note that the Pontryagin dual of $H^{1}(\Gamma_{n}, \varinjlim H^{1}(\Gamma_{m},W/T))[p]$ is  
  $X^{1}(F,T^{*}(1))^{\Gal(F_{\infty}/F_{n})}/p$. (Use the exact sequences  (1.3) and (1.5) p.10,11 in  \cite{PR}). As $X^{1}(F,T^{*}(1))$ is a finitely generated $\Lambda$-module,  
  the $\BZ_p$-modules $X^{1}(F,T^{*}(1))^{\Gal(F_{\infty}/F_{n})}$ have bounded rank (\cite[p.~11]{PR}). 
  
  We conclude that the dimensions $\dim_{k}H^{2}(\Gamma_{n},W/T)[p]
   $ are bounded.  
   \end{itemize}
      \end{proof}
  
  \section{Main result} 
    In this section we consider the vanishing of the second ordinary global Galois cohomology for adjoint over the cyclotomic tower. 
    
        Let the notation and hypothesis be as in \S\ref{s:Intro}-\S\ref{s:ordinary}.  
        \begin{prop} \label{free} 
         Suppose that 
   \begin{itemize}
  \item[(Aut)] $\bar{\rho}$ is automorphic,   
  \item[(ad$_{F(\zeta_{p})}$)] $\bar{\rho}|_{G_{F(\zeta_{p})}}$ is adequate (\cite[Def. 3.1.1]{Al}) and 
  \item[($\mu$)] $\mu(X^{1}(F,T^{*}(1))_{\tor})=0=\mu(X^{1}(F,T)_{\tor})$ for $T$ corresponding to $\Ad^{0}(\rho)$ with $\rho$ arising from an automorphic lift. 
   \end{itemize} 
         
        Then, $H^{1}_{\mathrm{ct}}(\Gamma_{0},\Ad^{0}\bar{\rho}(1) \otimes \Omega)$ is free over $\Omega$ of rank $[F:\BQ]$.
        \end{prop}
        \begin{proof}  We first show that $H^{1}_{\mathrm{ct}}(\Gamma_{0},\Ad^{0}\bar{\rho}(1) \otimes \Omega)$ is free as a $\Omega$-module.  It's enough to show that it's a $\Omega$-submodule of a free $\Omega$-module, since $\Omega$ is a PID. 
However,  the map      $$  H^{1}_{\mathrm{ct}}(\Gamma_{0},\Ad^{0}\bar{\rho}(1) \otimes \Omega) 
\rightarrow \bigoplus_{\fp|p}\underbrace{H^1_{\mathrm{ct}}(F_{\fp},  \Ad^{0}\bar{\rho}(1) \otimes \Omega)}_{\mbox{free by  Lemma \ref{sec:freeness} and \S\ref{ss:ct}}}$$
is injective:    Tate duality and Corollary \ref{mod p-WL-aut}
imply that  
$$\varprojlim \left( \mathrm{ker} \left( H^1(\Gamma_n, \Ad^{0}\bar{\rho}(1)) \rightarrow \bigoplus_{\fp|p} H^1(F_{n,\fp}, \Ad^{0}\bar{\rho}(1)) \right) \right)  = 0.$$
Indeed, the left hand side is dual to $\varinjlim \Sha^{2}(\Gamma_{n},\Ad^{0}\bar{\rho})$, which by Corollary \ref{mod p-WL-aut} vanishes as $\Gamma_{\infty}$ has cohomological dimension $1$.

 At this point we know that $H^{1}(\Gamma_{0},\Ad^{0}\bar{\rho}(1) \otimes \Omega)$ is   free of rank $r$, and must just show $r=[F:\BQ]$. 
   
  Note that, in view of the oddness of $\bar{\rho}$, the eigenvalues of complex conjugation
 on $\Ad^{0}\bar{\rho}$ are $-1, -1, +1$, and therefore the eigenvalues of complex conjugation on
 $\Ad^{0}\bar{\rho}(1)$ are $+1, +1, -1$. 
 By Tate's global Euler-Poincar\'{e} formula,  
  $$-\dim_{k} H^0(\Gamma_n, \Ad^{0}\bar{\rho}(1)) + \dim_{k} H^1(\Gamma_n, \Ad^{0}\bar{\rho}(1)) - \dim_{k} H^2(\Gamma_n, \Ad^{0}\bar{\rho}(1)) =   p^n [F:\BQ].$$  
 The first term is vanishing, and $\dim_{k} H^2(\Gamma_n, \Ad^{0}\bar{\rho}(1))$ remains bounded by Corollary \ref{mod p-WL-aut}.
 We conclude that there exists a constant $C$ such that  
\begin{equation} \label{cobound} \left| \dim_{k} H^1(\Gamma_n, \Ad^{0} \bar{\rho}(1))   -p^n [F:\BQ] \right| \leq C \end{equation} 

 As before we can identify $\Omega \simeq k[\![T]\!]$ in such a way
 that the quotient $k[\![T]\!]/(T^{p^n})$ is identified with the natural map $\Omega \rightarrow k[\Delta_n]$. Then 
from the sequence 
 $ 0 \rightarrow (T^{p^{n}}) \rightarrow k[\![T]\!] \rightarrow k[\![T]\!]/(T^{p^n}) \rightarrow 0$ we get
$$\underbrace{ H^1_{\mathrm{ct}}(\Gamma_0, \Ad^{0} \bar{\rho}(1) \otimes \Omega)/T^{p^n} }_{r p^n} \hookrightarrow \underbrace{  H^1_{\mathrm{ct}}(\Gamma_0, \Ad^{0} \bar{\rho}(1) \otimes k[\Delta_n])  }_{\approx [F:\BQ] p^n} \twoheadrightarrow
H^2_{\mathrm{ct}}(\Gamma_0, \Ad^{0}\bar{\rho}(1) \otimes \Omega)[T^{p^n}] $$
The final term is $\varprojlim H^2(\Gamma_m, \Ad^{0}\bar{\rho}(1))[T^{p^n}]$, and 
we saw in Corollary \ref{mod p-WL-aut} 
  that each term of the projective limit has dimension bounded above by $C$, thus the projective limit does too.

We conclude by comparing dimensions that $r= [F:\BQ]$.

        \end{proof} 
\begin{remark}
The freeness of $H^{1}_{\mathrm{ct}}(\Gamma_{0},\Ad^{0}\bar{\rho}(1) \otimes \Omega)$ as an $\Omega$-module
may be seen more directly: it's $\Omega$-torsion submodule is $(\Ad^{0}\ov{\rho}(1)\otimes \Omega)^{\Gamma_{\infty}}$ (cf.~\cite[p.~12]{PR}), which vanishes since $(\Ad^{0}\ov{\rho}(1))^{\Gamma_{0}}=0$ by our hypotheses.
\end{remark}
        
        We are ready for the main theorem:
 \begin{thm} \label{unob}
   Suppose that 
   \begin{itemize}
  \item[(Aut)] $\bar{\rho}$ is automorphic,   
  \item[(ad$_{F(\zeta_{p})}$)] $\bar{\rho}|_{G_{F(\zeta_{p})}}$ is adequate (\cite[Def. 3.1.1]{Al}) and 
  \item[($\mu$)] $\mu(X^{1}(F,T^{*}(1))_{\tor})=0=\mu(X^{1}(F,T)_{\tor})$ for $T$ arising from an automorphic lift. 
   \end{itemize} 
 Moreover, suppose that  $R_{\infty}$ is Noetherian.   Then  $\varinjlim H^2_{\ord}(\Gamma_n, \Ad^{0}\bar{\rho}) = 0$;
             in particular $$R_{\infty} \simeq W(k)[\![X_1, \dots, X_s]\!].$$ 
   \end{thm}
    \begin{proof}  To verify smoothness it is enough to check that a map $R_n \rightarrow A$ lifts to an infinitesimal extension
   $\tilde{A} \rightarrow A$ possibly after pullback via $R_m \rightarrow R_n$ for some $m > n$. 
   Equivalently, it is enough to verify  the vanishing of 
   $$\varinjlim H^2_{\ord}(\Gamma_n, \Ad^{0} \bar{\rho})$$

       By a duality argument, we have
       $$H^2_{\ord,n} \mbox{ is dual to }
       \ker(H^1(\Gamma_n, \Ad^{0} \bar{\rho}(1))  \rightarrow \bigoplus_{\fp|p} H^1(F_{n,\fp}, \Ad^{0} \bar{\rho}(1))/H^1_{\ord,*,n} )$$
    Moreover,   restriction maps for $H^2_{\ord, n}$ are identified with corestriction maps under the duality.    
    
    It remains to check that  
   $$\varprojlim_n \ker(H^1(\Gamma_n, \Ad^{0} \bar{\rho}(1))  \rightarrow \bigoplus_{\fp|p} H^1(F_{n,\fp}, \Ad^{0} \bar{\rho}(1))/H^1_{\ord,*,n} )$$
   (projective limit with respect to corestriction maps)   
   vanishes. Applying Shapiro's lemma as before (\S\ref{ss:ct}),  and noting that all the involved modules are finite and we can therefore
   commute cohomology and inverse limits (Mittag--Leffler) this is equivalent to checking  the injectivity of   
\begin{equation} \label{Cent} \underbrace{ H^1_{\mathrm{ct}}(\Gamma_0, \Ad^{0}(\bar{\rho})(1)\otimes \Omega )}_{\simeq \Omega^{[F:\BQ]} } \stackrel{\varphi}{ \rightarrow}  \bigoplus_{\fp|p}\underbrace{ H^1_{\mathrm{ct}}(F_{\fp}, \Ad^{0} \bar{\rho}(1) \otimes \Lambda)/H^1_{\ord, *}}_{k \oplus \Omega^{[F_{\fp}:\BQ_{p}]}} \end{equation}  
where we used the results of Proposition \ref{free}, Lemma \ref{spld}. 
    
    We will show that 
 $$ \mbox{$R_{\infty}$ Noetherian $\implies  \dim_k \mathrm{coker}(\varphi) < \infty$}$$
 which implies $\varphi$ is injective.     
 
   Now the cokernel of $\varphi$ is  
$$   \varprojlim_n   \mathrm{coker}(H^1(\Gamma_n, \Ad^{0} \bar{\rho}(1))  \rightarrow \bigoplus_{\fp|p}H^1(F_{n,\fp}, \Ad^{0} \bar{\rho}(1))/H^1_{\ord,*,n} ) $$
and  thus we deduce
$$\mathrm{coker}(\varphi)  \hookrightarrow \varprojlim H^2_{\ord,*,n} 
$$ 
where the group $H^2_{\ord,*,n}$ is defined  as in \S\ref{ss:tg}.

From Tate global duality $H^{2}_{\ord,*,n}\simeq (H^{1}_{\ord,n})^*$. Recall that 
$$\varinjlim H^1_{\ord}(\Gamma_n, \Ad^{0} \bar{\rho})$$
is isomorphic to the tangent space of $R_{\infty}$ : indeed, $ H^1_{\ord}(\Gamma_n, \Ad^{0} \bar{\rho})=\Hom(R_n, k[\varepsilon])$ and the tangent space of $R_{\infty}$, $\Hom(R_{\infty}, k[\varepsilon])$ is then the injective limit. 
In particular, it is finite-dimensional if $R_{\infty}$ is Noetherian.

We thus obtain
 $$\dim_{k} \mathrm{coker}(\varphi)  \leq \dim_k \varinjlim H^1_{\ord}(\Gamma_n, \Ad^{0} \bar{\rho})^{*}.$$

This concludes our argument.    
    
    \end{proof}

\end{document}